\documentclass[12pt,a4paper]{article}
\usepackage{amsmath}
\usepackage{amsfonts}
\usepackage{amssymb}
\usepackage{graphicx}
\usepackage{textcomp}
\usepackage{float}
\usepackage{algorithm}
\usepackage{algorithmic}
\usepackage{bm}

\usepackage[margin=2.54cm]{geometry}

\def\Dj{\hbox{D\kern-.73em\raise.30ex\hbox{-}
\raise-.30ex\hbox{}}}
\def\dj{\hbox{d\kern-.33em\raise.80ex\hbox{-}
\raise-.80ex\hbox{\kern-.40em}}}
\usepackage{graphicx}
\usepackage{amsmath,amsthm,amsfonts,amssymb,amscd,cite}
\allowdisplaybreaks

\newtheorem{theorem}{Theorem}
\newtheorem{theo}{Theorem}
\newtheorem{lemma}[theo]{Lemma}

\newtheorem{corollary}[theo]{Corollary}

\newtheorem{example}[theo]{Example}
\newtheorem{remark}[theo]{Remark}
\newenvironment{hproof}{%
  \proof}{\endproof}

\begin{document}

\baselineskip=0.30in

\vspace*{25mm}

\begin{center}
{\LARGE \bf An Efficient Algorithm for Latin Squares in a Bipartite Min-Max-Plus System}

\vspace{7mm}

{\large \bf Mubasher Umer}$^{1}$, {\large \bf Umar Hayat}$^1$, {\large \bf Fazal Abbas}$^{2}$\footnote{Corresponding author: fabbas1@stetson.edu }, {\large \bf Anurag Agarwal}$^{3}$, {\large \bf Petko Kitanov}$^{4}$

\vspace{6mm}

\baselineskip=0.18in

$^1${\it \small{Department of Mathematics, Quaid-i-Azam University, Islamabad, Pakistan}}\\
$^2${\it \small{Department of Mathematics and Computer Sciences, Stetson University, DeLand, FL, USA}}\\
$^3${\it \small{School of Mathematical Sciences, Rochester Institute of Technology, Rochester, NY, USA}}
$^4${\it \small{Wells College, Aurora, NY, USA}}
\vspace{6mm}

\end{center}

\begin{abstract}

 In this paper, we consider the eigenproblems for Latin squares in a bipartite min-max-plus system. The focus is upon developing a new algorithm to compute the eigenvalue and eigenvectors (trivial and non-trivial) for Latin squares in a bipartite min-max-plus system. We illustrate the algorithm using some examples. Furthermore, we compare the  results of our algorithm with some of the existing algorithms which shows that the propose method is more efficient.

\end{abstract}

\vspace{3mm}


\vspace{3mm}


\vspace{10mm}

\baselineskip=0.29in

\section{Introduction}

Time evolution of discrete event dynamic systems can be described through equations composed using three operations the maximum, the minimum and the addition. Such systems are described as min-max-plus systems. The bipartite min-max-plus systems are determined by the union of two sets of equations: one set of equations containing the maximization and the addition; and another set of equations containing the minimization and the addition.

Max-plus algebra has many applications in mathematics as well as in different areas such as mathematical physics, optimization, combinatorics, and algebraic geometry \cite{12}. Max-plus algebra is used in machine scheduling, telecommunication networks, control theory, manufacturing systems, parallel processing systems, and traffic control \cite{2,11,16}. Max-plus algebra is also used in image steganography \cite{7b}. In \cite{14}, the author has described the whole Dutch railway system by a max-plus system.

The eigenproblems for a matrix $A$ of order $n \times n$ are the problems to find out an eigenvalue $\lambda$ and the eigenvector $v$, such that $Av= \lambda v$. In this article, we consider eigenproblems for the bipartite min-max-plus systems. A power algorithm has been presented to calculate the eigenvalue and eigenvectors for such systems \cite{15}. M. Umer et al. \cite{17} develop an efficient algorithm to determine the eigenvalue and eigenvectors. In \cite{9}, the authors have demonstrated that the presence of eigenvectors (non-trivial) depends upon the position of the greatest and the least elements in the Latin squares in a bipartite min-max-plus systems. For further study of eigenproblems for discrete event systems, see \cite{5a,7a,8a,9,15}.

An algorithm is developed in this paper, that can calculate the eigenvectors corresponding to an eigenvalue $\lambda$. In particular, we apply this algorithm to find out the eigenvectors (trivial and non-trivial) for Latin squares in a bipartite min-max-plus systems.

The structure of the paper is along these lines. First of all, some preliminary notions and bipartite min-max-plus systems are presented in section $2$, then a robust algorithm is developed for determining the eigenvalue and eigenvectors of such systems. In section $3$, we present these systems for the Latin squares and calculate trivial and non-trivial eigenvectors for such models. The paper is concluded by presenting some conclusion and remarks in section $4$.

\vspace{10mm}

\section{Bipartite min-max-plus systems}

First of all define $\mathbb{R}_{\epsilon}=\mathbb{R}\cup \lbrace \epsilon \rbrace$, $\mathbb{R}_{\tau}=\mathbb{R}\cup\lbrace \tau \rbrace$, where $\epsilon=-\infty$, $\tau=+\infty$, and $\mathbb{R}$ is the set of real numbers. A matrix or a vector containing all components equal to $-\infty$ is represented by $\bm{\epsilon}$, whereas a matrix or a vector containing all components equal to $+\infty$ is represented by $\bm{\tau}$. $\underline{n}$ denotes the set of first $n$ positive integers. The following scalar operations are introduced as;

\begin{eqnarray*}
r\oplus s &=& max\lbrace r,s\rbrace, \text{ \ \ for each } r, s \in\mathbb{R}_{\epsilon} \\
r\oplus^{'} s &=& min\lbrace r,s\rbrace, \text{ \ \ for each } r, s \in\mathbb{R}_{\tau} \\
r\otimes s &=& r+s, \text{ \ \ \ \ \ \ \ \ for each } r, s \in\mathbb{R}.
\end{eqnarray*}

The algebraic structure $\mathbb{R}_{min}=(\mathbb{R}_{\tau},\oplus^{'},\otimes)$ represents min-plus algebra and $\mathbb{R}_{max}=(\mathbb{R}_{\epsilon},\oplus,\otimes)$ represents max-plus algebra. Since in both $\mathbb{R}_{min}$ and $\mathbb{R}_{max}$ the multiplication operator is defined by addition, therefore in both systems the notation for multiplication operator is the same.

The collection of all matrices of order $m\times n$ in min-plus and max-plus algebra is represented as $\mathbb{R}_{min}^{m\times n}$ and $\mathbb{R}_{max}^{m\times n}$ respectively. While
$\mathbb{R}_{min}^{m}$ and $\mathbb{R}_{max}^{m}$ denotes the set of all vectors in min-plus and max-plus algebra respectively. The scalar operations to matrices are extended as follows.

Suppose that $A=[a_{ij}]$, $B=[b_{ij}]$, $U=[u_{ij}]$, $V=[v_{ij}]$ such that $A, B \in \mathbb{R}_{max}^{m\times n}$; $U, V \in \mathbb{R}_{min}^{m\times n}$ and $\alpha \in \mathbb{R}$ then

\begin{eqnarray*}
A \oplus B &=& [c_{ij}], \text{ \ \ \ where \ } c_{ij} = max\lbrace a_{ij},b_{ij} \rbrace, \\
U \oplus^{'} V &=& [w_{ij}], \text{ \ \ where \ } w_{ij} = min\lbrace u_{ij},v_{ij} \rbrace, \\
\alpha \otimes A &=& \alpha \otimes [a_{ij}] =\alpha + [a_{ij}],
\end{eqnarray*}

\begin{center}
for $ i\in \underline{m} ,\: j\in \underline{n} $.
\end{center}

If $A \in \mathbb{R}_{max}^{m\times r}$, $B \in \mathbb{R}_{max}^{r\times n}$, $U \in \mathbb{R}_{min}^{m\times r}$ and $V \in \mathbb{R}_{min}^{r\times n}$, then

\begin{eqnarray*}
A \otimes B &=& [c_{ij}],  \text{ \ \ \ where \ } c_{ij} = \bigoplus_{k=1}^{r}(a_{ik}\otimes b_{kj}) = max\lbrace a_{ik} + b_{kj} \rbrace, \\
U \otimes V &=& [w_{ij}],  \text{ \ \ where \ } w_{ij} = {\bigoplus_{k=1}^{r}}^{'}(u_{ik}\otimes v_{kj}) = min\lbrace u_{ik} + v_{kj} \rbrace,
\end{eqnarray*}

\begin{center}
for $ \: k \in \underline{r}, i\in \underline{m} ,\: j\in \underline{n} $.
\end{center}

One can see that the notation of the multiplication operator in both systems is different. The addition is defined as maximum (minimum) in $\mathbb{R}_{max}$ (respectively, $\mathbb{R}_{min}$).
A bipartite min-max-plus system can be represented as;

\begin{eqnarray*}
u_{i}(l+1) &=& max \lbrace a_{i1}+w_{1}(l), . . . , a_{in}+w_{n}(l) \rbrace \\
w_{j}(l+1) &=& min \lbrace b_{j1}+u_{1}(l), . . . , b_{jm}+u_{m}(l) \rbrace,
\end{eqnarray*}

where $ u_{i}(l), a_{ij}\in \mathbb{R}_{\epsilon}$; $w_{j}(l), b_{ji}\in \mathbb{R}_{\tau}$; $l\in \mathbb{W}$ for all $ i\in \underline{m}$; $ j\in \underline{n}$. These equations can be written as;

\begin{eqnarray*}
u_{i}(l+1) &=& \bigoplus_{j=1}^{n}(a_{ij}\otimes w_{j}(l)), \\
w_{j}(l+1) &=& {\bigoplus_{i=1}^{m}}^{'}(b_{ji}\otimes u_{i}(l)).
\end{eqnarray*}

The above equations can be denoted as;

\begin{eqnarray}
    \left.\begin{array}{ll}
        u(l+1)=A \otimes w(l), \, \text{ \ \ }for\:  l\in\mathbb{W}\\
 w(l+1)=B \otimes^{'} u(l)  , \,  \text{ \ }for\:  l\in \mathbb{W}
        \end{array}\right\},
  \end{eqnarray}

where

\begin{center}
\[
u(l)=\begin{pmatrix} u_{1}(l)\\u_{2}(l)\\
\cdot\\
\cdot\\
\cdot\\
u_{m}(l) \end{pmatrix} \in \mathbb{R}_{\epsilon}^{m},
\text{ \ \ \ }
w(l)=\begin{pmatrix} w_{1}(l)\\w_{2}(l)\\
\cdot\\
\cdot\\
\cdot\\
w_{n}(l) \end{pmatrix} \in \mathbb{R}_{\tau}^{n}
\]
\end{center}

  \begin{center}
  $A=\begin{pmatrix}
    a_{11} & a_{12} & a_{13} & \dots  & a_{1n} \\
    a_{21} & a_{22} & a_{23} & \dots  & a_{2n} \\
    \vdots & \vdots & \vdots & \ddots & \vdots \\
    a_{m1} & a_{m2} & a_{m3} & \dots  & a_{mn}
\end{pmatrix} \in \mathbb{R}_{\epsilon}^{m\times n}$,
  \end{center}

  \begin{center}
 $B=\begin{pmatrix}
    b_{11} & b_{12} & b_{13} & \dots  & b_{1m} \\
    b_{21} & b_{22} & b_{23} & \dots  & b_{2m} \\
    \vdots & \vdots & \vdots & \ddots & \vdots \\
    b_{n1} & b_{n2} & b_{n3} & \dots  & b_{nm}
\end{pmatrix} \in \mathbb{R}_{\tau}^{n\times m}$.
  \end{center}

Compactly, the above system can be written by a mapping $\mathcal{M}(\cdot)$, such that

\begin{equation} \label{Eq 2}
x(l+1)=\mathcal{M} (x(l)),
\end{equation}

\begin{center}
where $x(l)=\begin{pmatrix} u(l)\\w(l)\end{pmatrix}, \mathcal{M} \begin{pmatrix} \begin{pmatrix} u(l)\\w(l)\end{pmatrix} \end{pmatrix} = \begin{pmatrix} A \otimes w(l) \\B \otimes^{'}u(l)
\end{pmatrix} ,$ for  $l\in \mathbb{W}.
$
\end{center}

For a system of type $(\ref{Eq 2})$, the notion of eigenvalue and eigenvectors is defined as follows. A real number $\lambda \in \mathbb{R}$ is said to be an eigenvalue corresponding to an eigenvector $v \in \mathbb{R}^{m+n}$ if

\begin{equation}
\mathcal{M}(v)=\lambda \otimes v.
\end{equation}

Define $A_{\lambda} = -\lambda \otimes A$ and $B_{\lambda} = -\lambda \otimes B$, corresponding to the eigenvalue $\lambda$ of a system of type $(\ref{Eq 2})$. Also define

\begin{equation}
\left.\begin{array}{ll}
u^{*}(l+1)=A_{\lambda} \otimes w^{*}(l)\\
w^{*}(l+1)=B_{\lambda} \otimes^{'} u^{*}(l)
\end{array}\right\}
\end{equation}

for $l \in\mathbb{W}$, where $u^{*}(l) \in \mathbb{R}_{\epsilon}^{m}$ and $w^{*}(l) \in \mathbb{R}_{\tau}^{n}$.

The above system can be denoted as,

\begin{equation} \label{Eq 5}
x^{*}(l+1)=\mathcal{N} (x^{*}(l))
\end{equation}

\begin{center}
where $x^{*}(l)=\begin{pmatrix} u^{*}(l)\\w^{*}(l)\end{pmatrix}, \mathcal{N} \begin{pmatrix} \begin{pmatrix} u^{*}(l)\\w^{*}(l)\end{pmatrix} \end{pmatrix} = \begin{pmatrix} A_{\lambda} \otimes w^{*}(l) \\B_{\lambda} \otimes^{'}u^{*}(l)
\end{pmatrix} ,$
\end{center}

for $l\in \mathbb{W}$. Let $v$ be a vector, then $\mathcal{N}^{l} (v)$ represents a vector obtained after applying $\mathcal{N}$ on the vector $v$ by $l$ times. A relation between $(\ref{Eq 2})$ and $(\ref{Eq 5})$ is shown in the following theorem.

\begin{theorem} \label{Th 1}
Let $\lambda$ be an eigenvalue for a system of type $(\ref{Eq 2})$ and \\$v = \begin{pmatrix} u(l)\\w(l)\end{pmatrix}$ be a vector, then
\begin{equation*}
\mathcal{M}(v) = \lambda \otimes \mathcal{N} (v).
\end{equation*}
\end{theorem}

\begin{proof}
Since
\begin{eqnarray*}
\mathcal{M} (v) &=& \begin{pmatrix} A \otimes w(l) \\B \otimes^{'}u(l)
   \end{pmatrix} \\ \vspace{2mm}
 &=& \lambda \otimes (-\lambda) \otimes \begin{pmatrix} A \otimes w(l) \\B \otimes^{'}u(l)
   \end{pmatrix} \\
 &=& \lambda \otimes \begin{pmatrix} (-\lambda) \otimes A \otimes w(l) \\(-\lambda) \otimes B \otimes^{'}u(l)
   \end{pmatrix} \\
 &=& \lambda \otimes \begin{pmatrix} A_{\lambda} \otimes w(l) \\B_{\lambda} \otimes^{'}u(l)
   \end{pmatrix} \\
 &=& \lambda \otimes \mathcal{N} (v).
\end{eqnarray*}
\end{proof}

\vspace{10mm}

\begin{corollary}
Let $v$ be an eigenvector corresponding to the eigenvalue $\lambda$ of a system of type $(\ref{Eq 2})$. Then
\begin{equation*}
\mathcal{N} (v) = v.
\end{equation*}
\end{corollary}

\begin{proof}
By definition, $\mathcal{M} (v) = \lambda \otimes v$. Using theorem one, we get, $\lambda \otimes \mathcal{N} (v) = \lambda \otimes v$. Hence $\mathcal{N} (v) = v$.
\end{proof}

\vspace{10mm}

\begin{theorem}
Let $\lambda$ be an eigenvalue of a system of type $(\ref{Eq 2})$ and let
\begin{equation*}
x^{*}(l+1) = \mathcal{N} (x^{*}(l))
\end{equation*}
for $l \in \mathbb{W}$, where $x^{*}(0)$ is an initial state vector. If $x^{*}(r) = x^{*}(s)$ for some integers $r>s\geq 0$, then
\begin{equation*}
\mathcal{M} (v) = \lambda \otimes v,
\end{equation*}
where
\begin{equation*}
v = x^{*}(s) \oplus ...  \oplus x^{*}(r-1).
\end{equation*}
\end{theorem}

\begin{proof}
From theorem $(\ref{Th 1})$,

\begin{eqnarray*}
\mathcal{M} (v) &=& \lambda \otimes \mathcal{N} (v) \\
 &=& \lambda \otimes \left\{ \begin{pmatrix} A_{\lambda} \otimes w^{*}(s) \\B_{\lambda} \otimes^{'}u^{*}(s)
   \end{pmatrix} \oplus . . . \oplus \begin{pmatrix} A_{\lambda} \otimes w^{*}(r-1) \\B_{\lambda} \otimes^{'}u^{*}(r-1)
   \end{pmatrix}\right\} \\
 &=& \lambda \otimes \{ x^{*}(s+1) \oplus ...  \oplus x^{*}(r) \} \\
 &=& \lambda \otimes v.
\end{eqnarray*}
\end{proof}

Now we present an algorithm in the following, which gives the eigenvectors corresponding to the eigenvalue $\lambda$.

\vspace{10mm}

\begin{algorithm}
\caption{Eigenvectors for systems of type $(\ref{Eq 2})$} \label{Alg 1}
\begin{enumerate}

\item Define $A_{\lambda} = -\lambda \otimes A$, and $B_{\lambda} = -\lambda \otimes B$.

\item Take an initial state vector $x^{*}(0)$.

\item Iterate $x^{*}(l+1) = \mathcal{N} (x^{*}(l)) = \begin{pmatrix} A_{\lambda} \otimes w^{*}(l) \\B_{\lambda} \otimes^{'}u^{*}(l) \end{pmatrix}$, for $l \in \mathbb{W}$, until there are positive integers $r>s\geq 0$, such that $x^{*}(r) = x^{*}(s)$.

\item Compute the eigenvector
\begin{equation*}
v = x^{*}(s) \oplus ...  \oplus x^{*}(r-1).
\end{equation*}

\item If $\mathcal{N} (v) = v$ then $v$ is the correct eigenvector and algorithm stops. Else if $\mathcal{N} (v) \neq v$ then go to the following step.

\item Take $v$ as new starting vector and iterate $x^{*}(l+1) = \mathcal{N} (x^{*}(l))$, until for some $t \geq 0$, it holds that $x^{*}(t+1) = x^{*}(t)$. Finally $x^{*}(t)$ is an eigenvector.

\end{enumerate}
\end{algorithm}

\vspace{20mm}

\begin{theorem}
Let $\lambda$ be an eigenvalue and $v$ be a vector computed as in the above algorithm $(\ref{Alg 1})$ for a system of type $(\ref{Eq 2})$, then

\begin{equation*}
\mathcal{N}(v) \geq v.
\end{equation*}
\end{theorem}

\begin{proof}
From the given algorithm

\begin{eqnarray*}
v &=& x^{*}(s) \oplus ...  \oplus x^{*}(r-1) \\
 &\geq& x^{*}(l) \text{ \ \ \ for all $l$ } \in \{s,...,r-1\}.
\end{eqnarray*}

Hence

\begin{eqnarray*}
\mathcal{N}(v) &\geq& \mathcal{N}(x^{*}(l)) \text{ \ \ \ for all $l$ } \in \{s,...,r-1\} \\
 &=& x^{*}(l+1) \text{ \ \ \ for all $l$ } \in \{s,...,r-1\}.
\end{eqnarray*}

These inequalities imply that

\begin{equation*}
\mathcal{N}(v) \geq x^{*}(s+1) \oplus . . . \oplus x^{*}(r) = v.
\end{equation*}

\end{proof}

\begin{lemma} \label{Le 1}
Let $v$ be a vector computed as in the above algorithm $(\ref{Alg 1})$ for a bipartite min-max-plus system. Let $(\ref{Eq 5})$ be restarted with $x^{*}(0)=v$, then

\begin{equation*}
x^{*}(l+1) \geq x^{*}(l) \text{ \ for all \ } l = 0, 1, 2, . . .
\end{equation*}
\end{lemma}

\begin{proof}
Since $x^{*}(l) = \mathcal {N}^{l} (v)$, which implies that
\begin{eqnarray*}
x^{*}(l+1) &=& \mathcal {N}^{l+1} (v) \\
 &=& \mathcal {N}^{l} ( \mathcal {N} (v) ) \\
 &\geq& \mathcal {N}^{l} (v) \\
 &=& x^{*}(l).
\end{eqnarray*}
\end{proof}

\begin{hproof}
Since the system ends up in a periodic behavior after a number of iterations, therefore step $3$ of the algorithm performs. If $\mathcal{N} (v) = v$ then $v$ is the correct eigenvector and clearly algorithm can stop. If $\mathcal{N} (v) \neq v$ then start by considering $y^{*}(0)$ as new initial vector and iterate $(\ref{Eq 5})$.

By first assumption, there exist integers $r>s\geq 0$, with $x^{*}(r) = x^{*}(s)$. By Lemma $(\ref{Le 1})$, $x^{*}(l+1) \geq x^{*}(l)$ for all $l \in \{s,...,r-1\}$. Our goal is to show that $x^{*}(l+1) = x^{*}(l)$ for all $l \in \{s,...,r-1\}$. Suppose on contrary that $x^{*}_{j^{\prime}}(l^{\prime}+1) > x^{*}_{j^{\prime}}(l^{\prime})$, for some integers $l^{\prime}, j^{\prime}$ such that, $s \leq l^{\prime} \leq r-1$ and  $1 \leq j^{\prime} \leq n$. We obtain that $x^{*}_{j^{\prime}}(r) > x^{*}_{j^{\prime}}(s)$. Because $x^{*}(r) = x^{*}(s)$, therefore $0 > 0$, a contradiction. Hence $x^{*}(l+1) = x^{*}(l)$ for all $l \in \{s,...,r-1\}$. Which completes the proof.
\end{hproof}

\section{Latin Squares in Bipartite (Min, Max, Plus)-systems}

A Latin square is a square matrix of order $n$, with elements from $n$ independent variables over $\mathbb{R}^{+}$ in such a way that each row and each column is a different permutation of the $n$ variables \cite{2a}. In the following, an example of a Latin square of order $4$ is given

\begin{center}
$L = \left(
  \begin{array}{cccc}
    3 & 2 & 4 & 1 \\
    4 & 1 & 3 & 2 \\
    2 & 4 & 1 & 3 \\
    1 & 3 & 2 & 4 \\
  \end{array}
\right).$
\end{center}
We consider the Latin squares of size $n$ in a system of type $(\ref{Eq 2})$. We have four possibilities for the Latin squares in a system of type $(\ref{Eq 2})$: $(1)$ The entries of both matrices $A$ and $B$ are $\underline{n}$; $(2)$ The entries of matrix $A$ are $\underline{n}_{\epsilon}$ and the entries of matrix $B$ are $\underline{n}$; $(3)$ The entries of matrix $A$ are $\underline{n}$ and the entries of matrix $B$ are $\underline{n}_{\tau}$; $(4)$ The entries matrix $A$ are $\underline{n}_{\epsilon}$ and the entries of matrix $B$ are $\underline{n}_{\tau}$.

In this section, we consider Latin squares for systems of type $(\ref{Eq 2})$ and algorithm $(\ref{Alg 1})$ is extended to calculate the eigenvalue and eigenvectors for such type of systems. In \cite{9}, authors show that for Latin squares in a system of type $(\ref{Eq 2})$, the eigenvalue $\lambda$ is determined as

\begin{equation} \label{Eq 3}
\lambda = \dfrac{max(A) + min(B)}{2}.
\end{equation}
Therefore eigenproblems for Latin squares in a system of type $(\ref{Eq 2})$ are more easy to answer. Now, we extend the algorithm $(\ref{Alg 1})$ for Latin squares in a system of type $(\ref{Eq 2})$ as follows:

\vspace{10mm}

\begin{algorithm}
\caption{Eigenvalue and Eigenvectors for Latin squares in a system of type $(\ref{Eq 2})$} \label{Alg 3}
\begin{enumerate}

\item Compute the eigenvalue as, $\lambda = \dfrac{max(A) + min(B)}{2}$.

\item Define $A_{\lambda} = -\lambda \otimes A$, and $B_{\lambda} = -\lambda \otimes B$.

\item Take a starting vector $x^{*}(0)$.

\item Iterate $(\ref{Eq 5})$, until for some integers $r>s\geq 0$, it holds that $x^{*}(r) = x^{*}(s)$.

\item Determine the eigenvector as,
\begin{equation*}
v = x^{*}(s) \oplus ...  \oplus x^{*}(r-1).
\end{equation*}

\item If $\mathcal{N} (v) = v$ then $v$ is the correct eigenvector corresponding to the eigenvalue $\lambda$ and algorithm can stops. Else if $\mathcal{N} (v) \neq v$ then go to the following step.

\item Take $v$ as initial state vector and iterate $x^{*}(l+1) = \mathcal{N} (x^{*}(l))$, until for some $t \geq 0$, it holds that $x^{*}(t+1) = x^{*}(t)$. Finally $x^{*}(t)$ is an eigenvector.

\end{enumerate}
\end{algorithm}

Here we find the eigenvalue and eigenvectors for Latin squares in systems of type $(\ref{Eq 2})$ by using algorithm $(\ref{Alg 3})$. First, recall the power algorithm for systems of type $(\ref{Eq 2})$, proposed by Subiono \cite{15}, then we compare it with algorithm $(\ref{Alg 3})$.

\vspace{10mm}

\begin{algorithm}
\caption{Eigenproblems for Bipartite (Min, Max, Plus)-Systems} $\label{Alg 2}$
\begin{enumerate}

\item Define a starting vector $x(0)$.

\item Iterate $(\ref{Eq 2})$, until there exist a real number $c$ and positive integers $r; s$, such that $r > s \geq 0$ with $x(r) = c \otimes x(s)$.

\item The eigenvalue is obtained as, $\lambda = \dfrac{c}{p-q}.$

\item Determine the eigenvector as, $v = \oplus^{p-q}_{j=1}(\lambda^{\otimes(p-q-j)} \otimes x(q+j-1)).$

\item If $\mathcal{M} (v) = \lambda \otimes v$ then $v$ is the required eigenvector corresponding to the eigenvalue $\lambda$ and algorithm can stop. If $\mathcal{M} (v) \neq \lambda \otimes v$, then the algorithm has to be continued as follows.

\item Define $x(0) = v$ as new starting vector and restart (\ref{Eq 2}), until for some $p$, there holds $x(p+1) = \lambda \otimes x(p)$. Then $x(p)$ is a correct eigenvector of system (\ref{Eq 2}).

\end{enumerate}
\end{algorithm}

\vspace{10mm}

To illustrate the algorithm $(\ref{Alg 3})$ and algorithm $(\ref{Alg 2})$, consider the following example. In this example, we consider a system of type $(\ref{Eq 2})$, where $A$ and $B$ are Latin Squares with entries in $\underline{n}_{\epsilon}$ and $\underline{n}_{\tau}$ respectively.

\vspace{10mm}

\begin{example} \label{Ex 1}
Let $A$ and $B$ are Latin Squares in a system of type $(\ref{Eq 2})$, given as follows:
\begin{center}
\[
 A=\begin{bmatrix}
    3 & 2 & \epsilon & 1 \\
    \epsilon & 1 & 3 & 2 \\
    2 & 3 & 1 & \epsilon \\
    1 & \epsilon & 2 & 3 \\
 \end{bmatrix},
 \text{ \ \ \ }
 B=\begin{bmatrix}
    2 & 3 & \tau & 1 \\
    3 & \tau & 1 & 2 \\
    1 & 2 & 3 & \tau \\
    \tau & 1 & 2 & 3 \\
 \end{bmatrix}.
 \]
 \end{center}

\begin{enumerate}
\item The eigenvalue $\lambda$ is given as,

\begin{equation*}
\lambda = \dfrac{max(A) + min(B)}{2} = 2.
\end{equation*}

By algorithm $(\ref{Alg 3})$,

\begin{center}
 \[
 A_{\lambda} =\begin{bmatrix}
    1 & 0 & \epsilon & -1 \\
    \epsilon & -1 & 1 & 0 \\
    0 & 1 & -1 & \epsilon \\
    -1 & \epsilon & 0 & 1 \\
 \end{bmatrix},
 \text{ \ \ \ }
 B_{\lambda} =\begin{bmatrix}
    0 & 1 & \tau & -1 \\
    1 & \tau & -1 & 0 \\
    -1 & 0 & 1 & \tau \\
    \tau & -1 & 0 & 1 \\
 \end{bmatrix}.
 \]
 \end{center}

Now, take the initial state vector
\begin{center}
\[
x^{*}(0)=\left(\begin{array}{c}
    u^{*}(0) \\
    w^{*}(0)
 \end{array}
\right)
\text{ \ with \ \ }
u^{*}(0)=\begin{pmatrix}
0\\
1\\
0\\
1 \end{pmatrix},
\text{ \ \ \ }
w^{*}(0)=\begin{pmatrix}
1\\
0\\
1\\
0 \end{pmatrix}.
\]
\end{center}

The following sequence is obtained, after iterating $(\ref{Eq 5})$,
\begin{eqnarray*}
x^{*}(0)\text{ \ }  \rightarrow \text{ \ \ } x^{*}(1)\text{ \ }  \rightarrow \text{ \ \ } x^{*}(2)\text{ \ } &\rightarrow\text{ \ } x^{*}(3)\text{ \ }  \rightarrow \text{ \ \ } x^{*}(4)\text{ \ }  \rightarrow \text{ \ \ } x^{*}(5)\text{ \ } &\rightarrow\text{ \ } x^{*}(6) \\
 \left(\begin{array}{c}
    0 \\
    1 \\
    0 \\
    1 \\
    1 \\
    0 \\
    1 \\
    0
  \end{array}
\right) \rightarrow \left(\begin{array}{c}
    2 \\
    2 \\
    1 \\
    1 \\
    0 \\
    -1 \\
    -1 \\
    0
  \end{array}
\right) \rightarrow \left(\begin{array}{c}
    1 \\
    0 \\
    0 \\
    1 \\
    0 \\
    0 \\
    1 \\
    1
  \end{array}
\right) &\rightarrow  \left(\begin{array}{c}
    1 \\
    2 \\
    1 \\
    2 \\
    0 \\
    -1 \\
    0 \\
    -1
  \end{array}
\right) \rightarrow \left(\begin{array}{c}
    1 \\
    1 \\
    0 \\
    0 \\
    1 \\
    0 \\
    0 \\
    1
  \end{array}
\right) \rightarrow \left(\begin{array}{c}
    2 \\
    1 \\
    1 \\
    2 \\
    -1 \\
    -1 \\
    0 \\
    0
  \end{array}
\right) &\rightarrow  \left(\begin{array}{c}
    0 \\
    1 \\
    0 \\
    1 \\
    1 \\
    0 \\
    1 \\
    0
  \end{array}
\right).
\end{eqnarray*}

Since $x^{*}(6) = x^{*}(0)$, therefore $s = 0, r = 6$. The required eigenvector $v$ is computed as,
\begin{eqnarray*}
v &=& \text{ \ }x^{*}(s) \oplus . . . \oplus x^{*}(r-1) \\
 &=& \text{ \ }x^{*}(0) \oplus . . . \oplus x^{*}(5) \\
 &=&
 \left(\begin{array}{c}
    0 \\
    1 \\
    0 \\
    1 \\
    1 \\
    0 \\
    1 \\
    0
  \end{array}
\right) \oplus \left(\begin{array}{c}
    2 \\
    2 \\
    1 \\
    1 \\
    0 \\
    -1 \\
    -1 \\
    0
  \end{array}
\right) \oplus \left(\begin{array}{c}
    1 \\
    0 \\
    0 \\
    1 \\
    0 \\
    0 \\
    1 \\
    1
  \end{array}
\right) \oplus  \left(\begin{array}{c}
    1 \\
    2 \\
    1 \\
    2 \\
    0 \\
    -1 \\
    0 \\
    -1
  \end{array}
\right) \oplus \left(\begin{array}{c}
    1 \\
    1 \\
    0 \\
    0 \\
    1 \\
    0 \\
    0 \\
    1
  \end{array}
\right) \oplus \left(\begin{array}{c}
    2 \\
    1 \\
    1 \\
    2 \\
    -1 \\
    -1 \\
    0 \\
    0
  \end{array}
\right)
 = \left(\begin{array}{c}
    2 \\
    2 \\
    1 \\
    2 \\
    1 \\
    0 \\
    1 \\
    1
  \end{array}
\right).
\end{eqnarray*}

Now verify that either $v$ is the correct eigenvector or not. So
\begin{equation*}
\mathcal{M}(v) = \left(\begin{array}{c}
                   4 \\
                   4 \\
                   3 \\
                   4 \\
                   3 \\
                   2 \\
                   3 \\
                   3
                 \end{array}
\right) = \lambda \otimes v.
\end{equation*}
Which shows that the eigenvector $v$ is the correct eigenvector obtained by algorithm $(\ref{Alg 3})$.

\item For algorithm $(\ref{Alg 2})$, take the initial vector

\begin{center}
\[
x(0)=\left(\begin{array}{c}
    u(0) \\
    w(0)
 \end{array}
\right)
\text{ \ with \ \ }
u(0)=\begin{pmatrix}
0\\
1\\
0\\
1 \end{pmatrix},
\text{ \ \ \ }
w(0)=\begin{pmatrix}
1\\
0\\
1\\
0 \end{pmatrix}.
\]
\end{center}

The following sequence is obtained, after iterating $(\ref{Eq 2})$
\begin{eqnarray*}
x(0)\text{ \ }  \rightarrow \text{ \ \ } x(1)\text{ \ }  \rightarrow \text{ \ \ } x(2)\text{ \ } &\rightarrow\text{ \ } x(3)\text{ \ }  \rightarrow \text{ \ \ } x(4)\text{ \ }  \rightarrow \text{ \ \ } x(5)\text{ \ } &\rightarrow\text{ \ } x(6) \\
 \left(\begin{array}{c}
    0 \\
    1 \\
    0 \\
    1 \\
    1 \\
    0 \\
    1 \\
    0
  \end{array}
\right) \rightarrow \left(\begin{array}{c}
    4 \\
    4 \\
    3 \\
    3 \\
    2 \\
    1 \\
    1 \\
    2
  \end{array}
\right) \rightarrow \left(\begin{array}{c}
    5 \\
    4 \\
    4 \\
    5 \\
    4 \\
    4 \\
    5 \\
    5
  \end{array}
\right) &\rightarrow  \left(\begin{array}{c}
    7 \\
    8 \\
    7 \\
    8 \\
    6 \\
    5 \\
    6 \\
    5
  \end{array}
\right) \rightarrow \left(\begin{array}{c}
    9 \\
    9 \\
    8 \\
    8 \\
    9 \\
    8 \\
    8 \\
    9
  \end{array}
\right) \rightarrow \left(\begin{array}{c}
    12 \\
    11 \\
    11 \\
    12 \\
    9 \\
    9 \\
    10 \\
    10
  \end{array}
\right) &\rightarrow  \left(\begin{array}{c}
    12 \\
    13 \\
    12 \\
    13 \\
    13 \\
    12 \\
    13 \\
    12
  \end{array}
\right).
\end{eqnarray*}
Since $x(6) = 12 \otimes x(0)$. It follows that $s = 0, r = 6$, and $c = 12$. The corresponding eigenvector $v$ is obtained as
\begin{eqnarray*}
v &=& \oplus^{r-s}_{j=1}(\lambda^{\otimes(r-s-j)} \otimes x(s+j-1)) \\
 &=& \oplus^{6}_{j=1}(\lambda^{\otimes(6-j)} \otimes x(j-1)) \\
 &=& \lambda^{\otimes(5)} \otimes x(0) \oplus \lambda^{\otimes(4)} \otimes x(1) \oplus \lambda^{\otimes(3)} \otimes x(2) \oplus \lambda^{\otimes(2)} \otimes x(3) \oplus \lambda \otimes x(4) \oplus x(5) \\
 &=& \left(\begin{array}{c}
    10 \\
    11 \\
    10 \\
    11 \\
    11 \\
    10 \\
    11 \\
    10
  \end{array}
\right) \oplus \left(\begin{array}{c}
    12 \\
    12 \\
    11 \\
    11 \\
    10 \\
    9 \\
    9 \\
    10
  \end{array}
\right) \oplus \left(\begin{array}{c}
    11 \\
    10 \\
    10 \\
    11 \\
    10 \\
    10 \\
    11 \\
    11
  \end{array}
\right) \oplus  \left(\begin{array}{c}
    11 \\
    12 \\
    11 \\
    12 \\
    10 \\
    9 \\
    10 \\
    9
  \end{array}
\right) \oplus \left(\begin{array}{c}
    11 \\
    11 \\
    10 \\
    10 \\
    11 \\
    10 \\
    10 \\
    11
  \end{array}
\right) \oplus \left(\begin{array}{c}
    12 \\
    11 \\
    11 \\
    12 \\
    9 \\
    9 \\
    10 \\
    10
  \end{array}
\right) \\
&=&  \left(\begin{array}{c}
    12 \\
    12 \\
    11 \\
    12 \\
    11 \\
    10 \\
    11 \\
    11
    \end{array}
\right).
\end{eqnarray*}
Which is the correct eigenvector.
\end{enumerate}
\end{example}

\begin{remark}
The main focus in this paper is to develop an efficient algorithm to find out the eigenvectors for systems of type $(\ref{Eq 2})$. Also, we have made numerical experiment to compare the algorithm $(\ref{Alg 3})$ and algorithm $(\ref{Alg 2})$.

It is clear by Example $(\ref{Ex 1})$, that for Latin squares in a system of type $(\ref{Eq 2})$, the computation of eigenvectors using algorithm $(\ref{Alg 3})$ is quit easy as compared to that of algorithm $(\ref{Alg 2})$. In case of algorithm $(\ref{Alg 2})$, it stops if there exist positive integers $r>s\geq 0$ and a real number $c$ with $x(r) = x(s) \otimes c$, while in that of algorithm $(\ref{Alg 3})$, it stops if for some integers $r>s\geq 0$ it holds $x^{*}(r) = x^{*}(s)$. Also in case of algorithm $(\ref{Alg 3})$, we obtain the eigenvector $v$ by a simple formula, given as $v = x^{*}(s) \oplus ...  \oplus x^{*}(r-1)$ however, while using that algorithm $(\ref{Alg 2})$ an eigenvector is obtained as
\begin{equation*}
v = \oplus_{j=1}^{r-s}(\lambda^{\otimes (r-s-j)} \otimes x(s+j-1)).
\end{equation*}
Which is quite difficult as compared to algorithm $(\ref{Alg 3})$.
\end{remark}

\section{CONCLUSION}
The eigenproblem of bipartite min-max-plus systems for Latin squares has been discussed in this work. An iterative algorithm was developed for the computation of the eigenvectors (trivial and nontrivial) for systems of type $(\ref{Eq 2})$. In particular, the computation of the eigenvalue and eigenvectors for Latin squares in a system of type $(2)$ has been made by the proposed algorithm. In the end, a computational comparison has been given. Similarly, one can derive an algorithm to calculate the eigenvalue and eigenvectors for separated min-max-plus systems.

\end{document}